\documentclass[12pt,leqno]{article}
\usepackage{amsfonts}
\usepackage{amsmath, amsthm, amsfonts, amssymb, color, ulem}
\usepackage{mathrsfs}
\newtheorem{thm}{Theorem}[section]

\newtheorem{lem}[thm]{Lemma}

\newtheorem{exa}[thm]{Example}
\theoremstyle{definition}

\newcommand{\scr}[1]{\mathscr #1}
\definecolor{wco}{rgb}{0.5,0.2,0.3}
\renewcommand{\bar}{\overline}
\numberwithin{equation}{section}
\newtheorem{rem}{Remark}[section]

\newcommand{\ua}{\uparrow}

\renewcommand{\hat}{\widehat}
\renewcommand{\tilde}{\widetilde}

\title{{\bf A note on strong convergence of implicit scheme for SDEs under local one-sided Lipschitz conditions\footnote{Supported by NSFC(No., 11561027, 11661039, 71371193), NSF of Jiangxi(No., 20161BAB211018), Scientific Research Fund of Jiangxi Provincial Education Department(No., GJJ150444).}
}}

\author{Li Tan\thanks{School of Statistics, Jiangxi University of Finance and Economics, Nanchang, Jiangxi, 330013, P. R. China. Email: tltanli@126.com.} \thanks{Research Center of Applied Statistics, Jiangxi University of Finance and Economics, Nanchang, Jiangxi, 330013, P. R. China.} \,  and \,  Chenggui Yuan\thanks{Department of Mathematics, Swansea University, Swansea, SA2 8PP, UK. Email: C.Yuan@swansea.ac.uk.}}

\begin{document}
\def\R{\mathbb R}  \def\ff{\frac} \def\ss{\sqrt} \def\B{\mathbf
B}
\def\N{\mathbb N} \def\kk{\kappa} \def\m{{\bf m}}
\def\dd{\delta} \def\DD{\Delta} \def\vv{\varepsilon} \def\rr{\rho}
\def\<{\langle} \def\>{\rangle} \def\GG{\Gamma} \def\gg{\gamma}
  \def\nn{\nabla} \def\pp{\partial} \def\EE{\scr E}
\def\d{\text{\rm{d}}} \def\bb{\beta} \def\aa{\alpha} \def\D{\scr D}
  \def\si{\sigma} \def\ess{\text{\rm{ess}}}
\def\beg{\begin} \def\beq{\begin{equation}}  \def\F{\scr F}
\def\Ric{\text{\rm{Ric}}} \def\Hess{\text{\rm{Hess}}}
\def\e{\text{\rm{e}}} \def\ua{\underline a} \def\OO{\Omega}  \def\oo{\omega}
 \def\tt{\tilde} \def\Ric{\text{\rm{Ric}}}
\def\cut{\text{\rm{cut}}} \def\P{\mathbb P} \def\ifn{I_n(f^{\bigotimes n})}
\def\C{\scr C}      \def\aaa{\mathbf{r}}     \def\r{r}
\def\gap{\text{\rm{gap}}} \def\prr{\pi_{{\bf m},\varrho}}  \def\r{\mathbf r}
\def\Z{\mathbb Z} \def\vrr{\varrho} \def\ll{\lambda}
\def\L{\scr L}\def\Tt{\tt} \def\TT{\tt}\def\II{\mathbb I}
\def\i{{\rm in}}\def\Sect{{\rm Sect}}\def\E{\mathbb E} \def\H{\mathbb H}
\def\M{\scr M}\def\Q{\mathbb Q} \def\texto{\text{o}} \def\LL{\Lambda}
\def\Rank{{\rm Rank}} \def\B{\scr B} \def\i{{\rm i}} \def\HR{\hat{\R}^d}
\def\to{\rightarrow}\def\l{\ell}
\def\8{\infty}\def\Y{\mathbb{Y}}

\maketitle

\begin{abstract}
Under a local one-sided Lipschitz condition,  Krylov \cite{KR} proved the existence and uniqueness of  the strong solutions for stochastic differential equations by using the  Euler-Maruyama approximation, where he showed that the sequence of numerical solutions converges to the true solution in probability as the stepsize tends to zero.
 In this note, we shall extend the results in \cite{KR} and investigate  an implicit numerical scheme for these equations under a local one-sided Lipschitz condition.

\end{abstract}


\section{Introduction}

The Euler-Maruyama (EM) scheme is a  basic numerical method for simulating solutions  for   stochastic differential equations (SDEs).  It plays an important role in numerical analysis since many SDEs do not have  an explicit solution.  Many papers have studied the  convergence and stability of the EM scheme for SDEs, and  most of the early works were concentrated on SDEs under a global Lipschitz condition ( see Kloeden and Platen \cite{kp92} for example ). Since the global Lipschitz condition is too strong for most equations,  weaker conditions have been  considered more recently. For example, Higham et. al \cite{hms02} studied strong convergence of the EM  method under a local Lipschitz condition.  They obtained   convergence results  if the SDEs satisfy the  local Lipschitz condition and  the $p$-th moments of the exact and numerical solutions are bounded for some $p>2$. Yuan and Mao \cite{ym08} improved the work of \cite{hms02} and provided convergence rate under a  local Lipschitz condition, where the local growth rate is logarithmic. Hutzenthaler et al. \cite{hjk12} proposed a tamed Euler scheme and obtained strong convergence of the scheme for SDEs with superlinearly growing and global one-sided Lipschitz continuous drift coefficients. Mao \cite{m03} established strong mean square convergence of  the EM scheme for stochastic functional differential equations under the local Lipschitz condition with the linear growth condition. Influenced by the classic truncated method, the truncated EM scheme was introduced in \cite{m15, m16} where convergence rates were obtained under local Lipschitz conditions.

As a  generalization of the EM scheme, the $\theta$-EM scheme has also been  widely considered. Mao and Szpruch \cite{ms13} dealt with strong convergence and stability of the $\theta$-EM scheme for SDEs with non-linear and non-Lipschitzian coefficients. Zong et. al \cite{zwh152} established convergence and stability of the $\theta$-EM scheme for stochastic differential delay equations with non-global Lipschitz continuous coefficients. Tan and Yuan \cite{ty162} studied the strong convergence and almost sure convergence for neutral stochastic differential delay equations under non-global Lipschitz continuous coefficients.

Most of the works mentioned above are concerned with SDEs with global or local Lipschitz conditions. However, for some equations, the local Lipschitz condition is too strong for the drift coefficient. For example,  $b(x)=x^3-|x|^{\frac{1}{2}}$  is not local Lipschitz, while we can show that $b(x)$ is local one-sided Lipschitz. There are also some works on non-local Lipschitz conditions. For example, Krylov \cite{KR} proved the existence and uniqueness of strong solutions for  SDEs  under a  local  one-sided Lipschitz condition by using the  EM approximation.  Gy\"{o}ngy \cite{g98} studied almost sure convergence for SDEs on domains of $\mathbb{R}^n$, and proved that if the drift satisfied the monotonicity condition and the diffusion coefficient was Lipschitz continuous, then the Euler scheme converged to the exact solution almost surely with rate less than 1/4. Gy\"{o}ngy and Sabanis \cite{gs13} extended the results of \cite{g98} to SDEs with delay, and showed that the rate of almost sure convergence under local Lipschitz condition is less than 1/2, while the rate is less than 1/4 under the  local one-sided Lipschitz condition. Later, Sabanis \cite{s13, s15} treated the tamed EM scheme of SDEs under global one-sided Lipschitz and local one-sided Lipschitz conditions, respectively. In their paper, strong convergence rates were obtained under a global one-sided Lipschitz condition while for the local one-sided Lipschitz cases, no convergence rates were obtained. To the best of our knowledge, there is no paper regarding the strong convergence of implicit numerical schemes under local one-sided Lipschitz conditions.  This leaves open the question of the  strong convergence of the  $\theta$-EM solutions under local one-sided Lipschitz conditions.  In this note, we shall extend the results in \cite{KR} and investigate  an implicit numerical scheme for these equations under a local one-sided Lipschitz condition.

Since the $\theta$-EM scheme is implicit, to guarantee the existence and uniqueness of numerical solutions, according to the theory of monotone operators, the drift coefficient is required to satisfy a global one-sided Lipschitz condition, while  we only  assume the drift coefficients satisfy a  local one-sided Lipschitz condition. Thus, in order to guarantee that the $\theta$-EM scheme is well defined, we modify the $\theta$-EM scheme with a truncated method and show that the modified $\theta$-EM solution converges strongly to the exact solution.

\section{Preliminaries}
Throughout the paper, we let $(\mathbb{R}^n, \<\cdot,\cdot\>, |\cdot|)$ be an $n$-dimensional Euclidean
space. Denote $\mathbb{R}^{n\times m}$ by the set of all $n\times m$ matrices endowed with Hilbert-Schmidt norm
 $\|A\|:=\sqrt{\mbox{trace}(A^*A)}$ for every $A\in \mathbb{R}^{n\times m}$, in which $A^*$ denotes the transpose of $A$.
 Let  $(\Omega, \mathscr{F}, \{\mathscr{F}_t\}_{t\ge 0}, \mathbb{P})$ be a complete
probability space,   let $\{W(t)\}_{t\ge0}$ be a $m$-dimensional Brownian motion defined on this probability space. In this paper, we consider the following SDE on $\mathbb{R}^n$:
\begin{equation}\label{brownian}
\begin{split}
\mbox{d}X(t)=&b(X(t))\mbox{d}t+\sigma(X(t))\mbox{d}W(t), t\ge 0,
\end{split}
\end{equation}
with initial data $X(0)=x_0$, where $b:\mathbb{R}^n\rightarrow\mathbb{R}^n$, $\sigma:\mathbb{R}^n\rightarrow\mathbb{R}^{n\times m}$.
In order to guarantee the existence and uniqueness of solutions to \eqref{brownian}, we shall assume that for any $x,y\in\mathbb{R}^n$,
\begin{enumerate}
\item[{\bf (A1)}] There exist positive constants $L$ and $l\ge 1$ such that
\begin{equation*}
\langle x,b(x)\rangle\vee\|\sigma(x)\|^2\le L(1+|x|^2),
\end{equation*}
and
\begin{equation*}
|b(x)|\le L(1+|x|^l).
\end{equation*}
\item[{\bf (A2)}] For any $R\ge 1$, there exists a positive constant $M_R\ge 1$ such that
\begin{equation*}
\langle x-y,b(x)-b(y)\rangle\vee\|\sigma(x)-\sigma(y)\|^2\le M_R|x-y|^2
\end{equation*}
for any $|x|\vee|y|\le R$.
\end{enumerate}

We call {\bf (A2)}  local one-sided Lipschitz condition.

\begin{exa}\label{example1}
{\rm Consider the following SDE on $\mathbb{R}$
\begin{equation*}
\mbox{d}X(t)=\left(X(t)+|X(t)|^2-X(t)^3-|X(t)|^{\frac{1}{2}}\right)\mbox{d}t+X(t)\mbox{d}W(t),
\end{equation*}
where $W$ is a scalar Brownian motion. Let $b(x)=x+|x|^2-x^3-|x|^{\frac{1}{2}}$ and $\sigma(x)=x$, we see that $b(x)$ is not local Lipschitz, and moreover, we can show that $b$ and $\sigma$ satisfy assumptions (A1)-(A2). Since it is obvious that $|b(x)|\le L(1+|x|^3)$, and by the Young inequality, we have
\begin{equation*}
\langle x,b(x)\rangle=x\left(x+|x|^2-x^3-|x|^{\frac{1}{2}}\right)\le\frac{5}{4}|x|^2+|x|^{\frac{3}{2}}\le2(1+|x|^2).
\end{equation*}
Besides, for any $R\ge 1$ and $|x|\vee|y|\le R$
\begin{equation*}
\begin{split}
\langle x-y,b(x)-b(y)\rangle=&(x-y)\left(x+|x|^2-x^3-|x|^{\frac{1}{2}}-y-|y|^2+y^3+|y|^{\frac{1}{2}}\right)\\
\le &(1+|x|+|y|)|x-y|^2-(x-y)(|x|^{\frac{1}{2}}-|y|^{\frac{1}{2}})\\
\le&(2R+1)|x-y|^2.
\end{split}
\end{equation*}
}
\end{exa}

\begin{lem}\label{exactpm}
{\rm Fix any $T>0$, let (A1)-(A2) hold, then there exists a unique solution $\{X(t)\}_{t\in[0,T]}$ to equation \eqref{brownian}, and the solution satisfies that
\begin{equation*}
\mathbb{E}\left(\sup\limits_{t\in[0,T]}|X(t)|^p\right)\le C,
\end{equation*}
for any $p\ge2$, where $C$ is a positive constant.
}
\end{lem}
Since the proof of this lemma is standard, we omit it here.

We now introduce $\theta$-EM scheme for \eqref{brownian}. Given any step size $\Delta$, define a partition $\{t_k=k\Delta:k=0,1,2,\cdots\}$ of the half line $t\ge 0$, let $y_0=x_0$ and define
\begin{equation}\label{discrete}
\begin{split}
y_{t_{k+1}}=&y_{t_k}+\theta b(y_{t_{k+1}})\Delta+(1-\theta) b(y_{t_{k}})\Delta+\sigma(y_{t_{k}})\Delta W_{t_k},
\end{split}
\end{equation}
where $\Delta W_{t_k}=W(t_{k+1})-W(t_k)$. Here $\theta\in [0,1]$ is an additional parameter that allows us to control the implicitness of the numerical scheme. For $\theta=0$, the $\theta$-EM scheme reduces to the EM scheme, and for $\theta=1$, it is  the backward EM scheme. Since the  $\theta$-EM scheme is an implicit method, we must guarantee that \eqref{discrete} is well defined. Generally speaking, for a given $y_{t_k}$,  to guarantee a unique solution $y_{t_{k+1}}$ for \eqref{discrete} is to assume that there exists a positive constant $M$ such that
\begin{equation*}
\langle x-y,b(x)-b(y)\rangle\le M|x-y|^2,
\end{equation*}
according to the theory of monotone operators (see \cite{zei} for more details). Moreover, as shown in Mao and Szpruch \cite{ms13}, this condition is somehow hard to relax. While in our assumption (A2), $M_R$ may tend to $\infty$ as $R\rightarrow\infty$. That is,  it is not certain if $\theta$-EM scheme \eqref{discrete} is well-defined under (A2). In the following, we will provide a new implicit numerical scheme, i.e.,  we will modify the $\theta$-EM scheme with a truncated method.

Choose a number $\Delta_1\in(0,1]$ and a strictly decreasing function $g:(0,\Delta_1]\rightarrow(0,\infty)$ such that for any $\Delta\in(0,\Delta_1]$,
\begin{equation}\label{gdelta}
g(\Delta_1)\ge 2,~~~~\lim\limits_{\Delta\rightarrow 0}g(\Delta)=\infty,~~M_{g(\Delta)}e^{M_{g(\Delta)}}\Delta^{\frac{1}{4}}\le 1,~~ {\rm and} (g(\Delta))^l\Delta\le 1.
\end{equation}

\begin{rem}
{\rm Condition \eqref{gdelta} can be satisfied. For example, let $g(\Delta)=\ln\Delta^{-\frac{1}{16}}$, obviously, $g(\Delta)$ is strictly decreasing and tends to $\infty$ as $\Delta\rightarrow0$. Set $\Delta_1=e^{-32}$, $M_{g(\Delta)}=g(\Delta)$, then for $\Delta\in(0,e^{-32}]$, $g(\Delta)\ge2$ and $M_{g(\Delta)}e^{M_{g(\Delta)}}\Delta^{\frac{1}{4}}\le\Delta^{\frac{1}{8}}< 1$ and $(g(\Delta))^l\Delta\le 1$.
}
\end{rem}

For $\Delta\in(0,\Delta_1]$, define a smooth, non-negative function with compact support $\zeta_{\Delta}\in\mathcal{C}_c^\infty(\mathbb{R}^n)$such that
\begin{equation*}
\zeta_\Delta(x)=\left\{
\begin{array}{lll}
&1,~~&{\rm for}~|x|\le g(\Delta)-1,\\
&0,~~&{\rm for}~|x|>g(\Delta).
\end{array}
\right.
\end{equation*}
and $\zeta_{\Delta}(x)\le 1$ for all $x\in\mathbb{R}^n$. It is obvious that $\zeta_{\Delta}(x)$ is Lipschitz with some constant $C_{\zeta_{\Delta}}$.

Define the following truncated functions
\begin{equation*}
b_\Delta(x)=\zeta_{\Delta}(x)b(x).
\end{equation*}
We have the following results:

\begin{lem}\label{lemmono}
{\rm Let (A1) hold, then for any $x\in\mathbb{R}^n$
\begin{equation*}
\langle x, b_\Delta(x)\rangle\le L(1+|x|^2)
\mbox{ and }
|b_\Delta(x)|\le L(1+|x|^l).
\end{equation*}

}
\end{lem}

\begin{lem}\label{globalone}
{\rm Let (A1)-(A2) hold, then for any $x,y\in\mathbb{R}^n$,
\begin{equation*}
\langle x-y,b_\Delta(x)-b_\Delta(y)\rangle\le \bar{M}_{g(\Delta)}|x-y|^2,
\end{equation*}
where $\bar{M}_{g(\Delta)}=M_{g(\Delta)}+C_{\zeta_{\Delta}}L[1+(g(\Delta))^l]$.
}
\end{lem}
\begin{proof}
{\bf Case 1:} First, we assume  $|x|\le g(\Delta)$ and $|y|\le g(\Delta)$, then
$0 \le \zeta_{\Delta}(x)\le 1, 0 \le \zeta_{\Delta}(y)\le 1$.  Note that
\begin{equation*}
\begin{split}
&\langle x-y, b_\Delta(x)-b_\Delta(y)\rangle=\langle x-y, \zeta_{\Delta}(x)b(x)-\zeta_{\Delta}(y)b(y)\rangle\\
=&\zeta_{\Delta}(x)\langle x-y,b(x)-b(y)\rangle+\langle x-y, (\zeta_{\Delta}(x)-\zeta_{\Delta}(y))b(y)\rangle\\
=:&I_1+I_2.
\end{split}
\end{equation*}
The  local one-sided Lipschitz condition (A2) implies
\begin{equation*}
\begin{split}
I_1\le M_{g(\Delta)}|x-y|^2.
\end{split}
\end{equation*}
Since for $|y|\le g(\Delta)$, by (A1) we have $|b(y)|\le L[1+(g(\Delta))^l]$, $\zeta_{\Delta}$ is also Lipschitz, thus we have
\begin{equation*}
\begin{split}
I_2\le C_{\zeta_{\Delta}}L[1+(g(\Delta))^l]|x-y|^2.
\end{split}
\end{equation*}
Therefore $b_\Delta$ satisfies the global one-sided Lipschitz condition with constant
\begin{equation*}
\begin{split}
\bar{M}_{g(\Delta)}:=M_{g(\Delta)}+C_{\zeta_{\Delta}}L[1+(g(\Delta))^l].
\end{split}
\end{equation*}
{\bf Case 2:} $|x|\le g(\Delta)$ and $|y|> g(\Delta),$  we have
\begin{equation*}
\begin{split}
&\langle x-y, b_\Delta(x)-b_\Delta(y)\rangle=\langle x-y, \zeta_{\Delta}(x)b(x)\rangle\\
&=\langle x-y, (\zeta_{\Delta}(x)-\zeta_{\Delta}(y))b(x)\rangle\le  C_{\zeta_{\Delta}}L[1+(g(\Delta))^l]|x-y|^2.
\end{split}
\end{equation*}
This means $b_\Delta$ satisfies the global one-sided Lipschitz condition with constant
\begin{equation*}
\begin{split}
\bar{M}_{g(\Delta)}:=C_{\zeta_{\Delta}}L[1+(g(\Delta))^l].
\end{split}
\end{equation*}
{\bf Case 3:} $|x|> g(\Delta)$ and $|y|> g(\Delta).$  This is trivial, since
\begin{equation*}
\begin{split}
\langle x-y, b_\Delta(x)-b_\Delta(y)\rangle=0.
\end{split}
\end{equation*}
The proof is therefore complete.
\end{proof}

Now we define the corresponding modified $\theta$-EM scheme as follows:
\begin{equation}\label{discrete22}
\begin{split}
y_{t_{k+1}}=&y_{t_k}+\theta b_\Delta(y_{t_{k+1}})\Delta+(1-\theta) b_\Delta(y_{t_{k}})\Delta+\sigma(y_{t_{k}})\Delta W_{t_k}.
\end{split}
\end{equation}
By \eqref{gdelta} and Lemma \ref{globalone}, the theory of monotone operators implies that \eqref{discrete22} is well defined.
Due to the implicitness of $\theta$-EM scheme, we also require $\theta\Delta<\frac{1}{2 L}$, where $L$ is defined as in assumption (A1). Thus in the following sections,  we will set a $\Delta_2\in\left(0,\frac{1}{2 \theta L}\right)$, and choose the stepsize such that for $\theta=0,$ let  $\Delta\in(0,\Delta_1]$, for $\theta\in (0, 1]$,  let $\Delta\in(0,\Delta_1\wedge\Delta_2]$.

We find it is convenient to work with a continuous form of a numerical method. Fix any $T>0$, denote $M=\lfloor\frac{T}{\Delta}\rfloor$. Let $Y_\Delta(t)$ denote  the corresponding continuous form of $y_{t_k}$ such that $Y_\Delta(0)=x_0$, and for any $t\in[0,T]$, define
\begin{equation}\label{continuous}
\mbox{d}[Y_\Delta(t)-\theta b_\Delta(Y_\Delta(t))\Delta]=b_\Delta(\bar{Y}_\Delta(t))\mbox{d}t+\sigma(\bar{Y}_\Delta(t))\mbox{d}W(t),
\end{equation}
where $\bar{Y}_\Delta(t)$ is defined by
\begin{equation*}
\bar{Y}_\Delta(t):=y_{t_k} \quad \mbox{for} \quad t\in[t_k, t_{k+1}).
\end{equation*}
 It can be verified that $Y_\Delta(t_k)=y_{t_k}$, $k=0,1,\cdots, M$.

\begin{lem}\label{pmoment}
{\rm Let (A1) hold. Then for $\theta\in[\frac{1}{2},1]$, there exists a positive constant $C$ independent of $\Delta$  such that for $p\ge2$,
\begin{equation*}
\begin{split}
\mathbb{E}\left(\sup\limits_{0\le k\le M}|y_{t_k}|^{p}\right)\le C.
\end{split}
\end{equation*}
}
\end{lem}
\begin{proof}
Denote $z_{t_k}=y_{t_k}-\theta b_\Delta(y_{t_k})\Delta$, by \eqref{discrete22}, we deduce that
\begin{equation*}
\begin{split}
|z_{t_{k+1}}|^2=&|z_{t_k}|^2+2\langle z_{t_k},b_\Delta(y_{t_k})\Delta\rangle+|b_\Delta(y_{t_k})|^2\Delta^2+|\sigma(y_{t_k})\Delta W_{t_k}|^2\\
&+2\langle z_{t_k}+b_\Delta(y_{t_k})\Delta,\sigma(y_{t_k})\Delta W_{t_k}\rangle\\
=&|z_{t_k}|^2+2\langle y_{t_k},b_\Delta(y_{t_k})\Delta\rangle+(1-2\theta)|b_\Delta(y_{t_k})|^2\Delta^2+|\sigma(y_{t_k})\Delta W_{t_k}|^2\\
&+2\langle y_{t_k}+(1-\theta)b_\Delta(y_{t_k})\Delta,\sigma(y_{t_k})\Delta W_{t_k}\rangle.
\end{split}
\end{equation*}
Noting that $\theta\ge\frac{1}{2}$ and  $b_\Delta(y_{t_k})\Delta=\frac{1}{\theta}(y_{t_k}-z_{t_k})$, and using Lemma \ref{lemmono} yields
\begin{equation*}
\begin{split}
|z_{t_{k+1}}|^2\le&|z_{t_k}|^2+2\Delta\langle y_{t_k},b_\Delta(y_{t_k})\rangle+|\sigma(y_{t_k})\Delta W_{t_k}|^2\\
&+\frac{2}{\theta}\langle y_{t_k},\sigma(y_{t_k})\Delta W_{t_k}\rangle-2\frac{1-\theta}{\theta}\langle z_{t_k},\sigma(y_{t_k})\Delta W_{t_k}\rangle\\
\le&|z_{t_k}|^2+2L\Delta(1+|y_{t_k}|^2)+|\sigma(y_{t_k})\Delta W_{t_k}|^2\\
&+\frac{2}{\theta}\langle y_{t_k},\sigma(y_{t_k})\Delta W_{t_k}\rangle-2\frac{1-\theta}{\theta}\langle z_{t_k},\sigma(y_{t_k})\Delta W_{t_k}\rangle.
\end{split}
\end{equation*}
Summing both sides from 0 to $k$, we get
\begin{equation*}\label{ztk-Dztk}
\begin{split}
|z_{t_{k+1}}|^2\le&|z_{t_0}|^2+2LT+2L\Delta\sum\limits_{i=0}^k|y_{t_i}|^2+\sum\limits_{i=0}^k|\sigma(y_{t_i})\Delta W_{t_i}|^2\\
&+\frac{2}{\theta}\sum\limits_{i=0}^k\langle y_{t_i},\sigma(y_{t_i})\Delta W_{t_i}\rangle-2\frac{1-\theta}{\theta}\sum\limits_{i=0}^k\langle z_{t_i},\sigma(y_{t_i})\Delta W_{t_i}\rangle.
\end{split}
\end{equation*}
By the elementary inequality $\left|\sum\limits_{i=1}^nx_i\right|^p\le n^{p-1}\sum\limits_{i=1}^n|x_i|^p, \,  p\ge 1$, we have
\begin{equation}\label{visa}
\begin{split}
&|z_{t_{k+1}}|^{2p}\le5^{p-1}2^pL^p\Delta^p\left(\sum\limits_{i=0}^k|y_{t_i}|^2\right)^{p}+5^{p-1}\left(\sum\limits_{i=0}^k|\sigma(y_{t_i})\Delta W_{t_i}|^2\right)^p\\
&+5^{p-1}4^p\left|\sum\limits_{i=0}^k\langle y_{t_i},\sigma(y_{t_i})\Delta W_{t_i}\rangle\right|^p
+5^{p-1}2^p\left|\sum\limits_{i=0}^k\langle z_{t_i},\sigma(y_{t_i})\Delta W_{t_i}\rangle\right|^p +5^{p-1}(|z_{t_0}|^2+2LT)^{p}.
\end{split}
\end{equation}
For $0<j<M$, it is easy to observe that
\begin{equation*}
\begin{split}
\mathbb{E}\left[\sup\limits_{0\le k\le j}\left(\sum\limits_{i=0}^k|y_{t_i}|^2\right)^{p}\right]\le M^{p-1}\sum\limits_{i=0}^j\mathbb{E}|y_{t_i}|^{2p}.
\end{split}
\end{equation*}
By Lemma \ref{lemmono}, we compute
\begin{equation*}
\begin{split}
&\mathbb{E}\left[\sup\limits_{0\le k\le j}\left(\sum\limits_{i=0}^k|\sigma(y_{t_i})\Delta W_{t_i}|^2\right)^p\right]
\le  M^{p-1}\mathbb{E}\left(\sum\limits_{i=0}^j\|\sigma(y_{t_i})\|^{2p}|\Delta W_{t_i}|^{2p}\right)\\
\le& M^{p-1}\sum\limits_{i=0}^j\mathbb{E}\|\sigma(y_{t_i})\|^{2p}\mathbb{E}|\Delta W_{t_i}|^{2p}
\le(2p-1)!!M^{p-1}L^p\Delta^p\sum\limits_{i=0}^j\mathbb{E}(1+|y_{t_i}|^2)^{p}\\
\le&(2p-1)!!(2ML)^{p}\Delta^{p}+2(2p-1)!!(2M)^{p-1}L^p\Delta^p\sum\limits_{i=0}^j\mathbb{E}|y_{t_i}|^{2p}.
\end{split}
\end{equation*}
With Lemma \ref{lemmono}, the Young inequality and the Burkholder-Davis-Gundy (BDG) inequality, we arrive at
\begin{equation*}
\begin{split}
&\mathbb{E}\left[\sup\limits_{0\le k\le j}\left|\sum\limits_{i=0}^k\langle y_{t_i},\sigma(y_{t_i})\Delta W_{t_i}\rangle\right|^p\right]
\le c_1(p)\mathbb{E}\left(\sum\limits_{i=0}^j|y_{t_i}|^2\|\sigma(y_{t_i})\|^{2}\Delta\right)^{\frac{p}{2}}\\
\le& c_1(p)M^{\frac{p}{2}-1}L^{\frac{p}{2}}\Delta^{\frac{p}{2}}\mathbb{E}\sum\limits_{i=0}^j|y_{t_i}|^p(1+|y_{t_i}|^2)^{\frac{p}{2}}\\
\le&c_1(p)2^{p-2}(2M)^{\frac{p}{2}-1}L^{\frac{p}{2}}\Delta^{\frac{p}{2}}M
+5c_1(p)2^{p-2}(2M)^{\frac{p}{2}-1}L^{\frac{p}{2}}\Delta^{\frac{p}{2}}\sum\limits_{i=0}^j\mathbb{E}|y_{t_i}|^{2p}.
\end{split}
\end{equation*}
where $c_1(p)=\left[\frac{p^{p+1}}{2(p-1)^{p-1}}\right]^{p/2}$. Similarly, with Lemma \ref{lemmono}, the Young inequality and the BDG inequality again
\begin{equation*}
\begin{split}
&\mathbb{E}\left[\sup\limits_{0\le k\le j}\left|\sum\limits_{i=0}^k\langle z_{t_i},\sigma(y_{t_i})\Delta W_{t_i}\rangle\right|^p\right]
\le c_1(p)\mathbb{E}\left(\sum\limits_{i=0}^j|z_{t_i}|^2\|\sigma(y_{t_i})\|^{2}\Delta\right)^{\frac{p}{2}}\\
\le&c_1(p)M^{\frac{p}{2}-1}L^{\frac{p}{2}}\Delta^{\frac{p}{2}}\mathbb{E}\left(\sup\limits_{0\le k\le j+1}|z_{t_k}|^p\sum\limits_{i=0}^j(1+|y_{t_i}|^2)^{\frac{p}{2}}\right)\\
\le&\frac{1}{2}\mathbb{E}\left(\sup\limits_{0\le k\le j+1}|z_{t_k}|^{2p}\right)+\frac{1}{4}c^2_1(p)(2M)^{p-1}L^{p}\Delta^{p}
+\frac{1}{2}c^2_1(p)(2M)^{p-1}L^{p}\Delta^{p}\sum\limits_{i=0}^j\mathbb{E}|y_{t_i}|^{2p}.
\end{split}
\end{equation*}
Noting $M\Delta\le T$ and sorting this inequalities together, we derive from \eqref{visa} that
\begin{equation}\label{ztk1-Dztk1}
\begin{split}
\mathbb{E}\left(\sup\limits_{0\le k\le j+1}|z_{t_{k}}|^{2p}\right)
\le&C+C\sum\limits_{i=0}^j\mathbb{E}|y_{t_i}|^{2p}.
\end{split}
\end{equation}
Since $y_{t_k}=z_{t_k}+\theta b_\Delta(y_{t_k})\Delta$, we deduce from Lemma \ref{lemmono} that
\begin{equation}\label{what}
\begin{split}
|z_{t_{k}}|^{2}=&|y_{t_k}|^2+\theta^2\Delta^2|b_\Delta(y_{t_k})|^2-2\theta\Delta\langle y_{t_k},b_\Delta(y_{t_k})\rangle\\
\ge&|y_{t_k}|^2-2\theta L\Delta(1+|y_{t_k}|^2)=(1-2\theta L\Delta)|y_{t_k}|^2-2\theta L\Delta,
\end{split}
\end{equation}
this implies
\begin{equation*}
\begin{split}
|y_{t_k}|^2\le&\frac{1}{1-2\theta L\Delta}(|z_{t_{k}}|^{2}+2\theta L\Delta).
\end{split}
\end{equation*}
Thus, we arrive at  by \eqref{ztk1-Dztk1}
\begin{equation*}
\begin{split}
\mathbb{E}\left(\sup\limits_{0\le k\le j+1}|y_{t_k}|^{2p}\right)\le C+C\mathbb{E}\left(\sup\limits_{0\le k\le j+1}|z_{t_{k}}|^{2p}\right)
\le C+C \sum_{k=0}^j\mathbb{E}\left(\sup\limits_{0\le i\le k}|y_{t_i}|^{2p}\right).
\end{split}
\end{equation*}
Finally, the desired result follows from the discrete Gronwall inequality.
\end{proof}

\begin{lem}\label{ytytk}
{\rm
Let (A1) hold. Then it holds that for any $p\ge 2$ and $\theta\in[\frac{1}{2},1]$,
\begin{equation*}
\mathbb{E}\left(\sup\limits_{0\le t \le T}|Y_\Delta(t)|^p\right)\le C,
\end{equation*}
and
\begin{equation*}
\mathbb{E}\left[\sup\limits_{0\le k\le M-1}\sup\limits_{t_k\le t<t_{k+1}}|Y_\Delta(t)-Y_\Delta(t_k)|^p\right]\le C\Delta^{\frac{p}{2}},
\end{equation*}
where $C$ is a constant independent of $\Delta$.
}
\end{lem}
\begin{proof}
Denote by $Z_\Delta(t)=Y_\Delta(t)-\theta b_\Delta(Y_\Delta(t))\Delta$, for any $p\ge 2$ we have
\begin{equation*}
\begin{split}
\mathbb{E}\left(\sup\limits_{0\le u\le t}|Z_\Delta(u)|^p\right)\le &3^{p-1}|Z_\Delta(0)|^p+3^{p-1}\mathbb{E}\left(\sup\limits_{0\le u\le t}\left|\int_0^u b_\Delta(\bar{Y}_\Delta(s))\mbox{d}s\right|^p\right)\\
&+3^{p-1}\mathbb{E}\left(\sup\limits_{0\le u\le t}\left|\int_{0}^u\sigma(\bar{Y}_\Delta(s))\mbox{d}W(s)\right|^p\right),
\end{split}
\end{equation*}
where $Z_\Delta(0)=x_0-\theta b_\Delta(x_0)\Delta$. Using the H\"{o}lder inequality, the BDG inequality, and together with Lemmas \ref{lemmono} and \ref{pmoment} yields
\begin{equation}\label{supcontin}
\begin{split}
\mathbb{E}\left(\sup\limits_{0\le u\le t}|Z_\Delta(u)|^p\right)
\le &3^{p-1}|Z_\Delta(0)|^p+3^{p-1}t^{p-1}\mathbb{E}\int_0^t\left|b_\Delta(\bar{Y}_\Delta(s))\right|^p\mbox{d}s\\
&+3^{p-1}c_1(p)\mathbb{E}\left(\int_{0}^t\|\sigma(\bar{Y}_\Delta(s))\|^2\mbox{d}s\right)^{\frac{p}{2}}\\
\le& C+C\mathbb{E}\int_0^t[1+|\bar{Y}_\Delta(s)|^p+|\bar{Y}_\Delta(s)|^{lp}]\mbox{d}s\le C.
\end{split}
\end{equation}
In the same way as in \eqref{what},  we derive
\begin{equation*}
\begin{split}
|Y_\Delta(t)|^2\le&\frac{1}{1-2\theta L\Delta}(|Z_\Delta(t)|^{2}+2\theta L\Delta).
\end{split}
\end{equation*}
Consequently, we derive from \eqref{supcontin} that
\begin{equation*}
\begin{split}
\mathbb{E}\left(\sup\limits_{0\le t\le T}|Y_\Delta(t)|^p\right)\le& C+C\mathbb{E}\left(\sup\limits_{0\le t\le T}|Z_\Delta(t)|^p\right)\le C.
\end{split}
\end{equation*}
We now give a proof for the second part. By \eqref{continuous}, we have
\begin{equation*}
\begin{split}
&\mathbb{E}\left(\sup\limits_{t_k\le t<t_{k+1}}|Z_\Delta(t)-Z_\Delta(t_k)|^p\right)\le2^{p-1}\mathbb{E}\left(\sup\limits_{t_k\le t<t_{k+1}}\left|\int_{t_k}^t b_\Delta(\bar{Y}_\Delta(s))\mbox{d}s\right|^p\right)\\
&+2^{p-1}\mathbb{E}\left(\sup\limits_{t_k\le t<t_{k+1}}\left|\int_{t_k}^t\sigma(\bar{Y}_\Delta(s))\mbox{d}W(s)\right|^p\right).
\end{split}
\end{equation*}
Then, by  Lemmas \ref{lemmono} and \ref{pmoment}, the H\"{o}lder inequality, and the BDG inequality, we arrive at
\begin{equation}\label{phizt}
\begin{split}
&\mathbb{E}\left(\sup\limits_{t_k\le t<t_{k+1}}|Z_\Delta(t)-Z_\Delta(t_k)|^p\right)
\le 2^{p-1}\Delta^{p-1}\mathbb{E}\int_{t_k}^{t_{k+1}}\left|b_\Delta(\bar{Y}_\Delta(s))\right|^p\mbox{d}s\\
&+2^{p-1}c_1(p)\mathbb{E}\left[\int_{t_k}^{t_{k+1}}\left\|\sigma(\bar{Y}_\Delta(s))\right\|^2\mbox{d}s\right]^{\frac{p}{2}}
\le  C\Delta^{p}+C\Delta^{\frac{p}{2}}\le C\Delta^{\frac{p}{2}}.
\end{split}
\end{equation}
Hence, Lemma \ref{pmoment} and \eqref{phizt} give that
\begin{equation*}
\begin{split}
&\mathbb{E}\left(\sup\limits_{t_k\le t<t_{k+1}}|Z_\Delta(t)-\bar{Y}_\Delta(t)|^p\right)
\le2^{p-1}\mathbb{E}\left(\sup\limits_{t_k\le t<t_{k+1}}|Z_\Delta(t)-Z_\Delta(t_k)|^p\right)\\
&+2^{p-1}\theta^p\Delta^p\mathbb{E}\left(\sup\limits_{t_k\le t<t_{k+1}}|b_\Delta(\bar{Y}_\Delta(t)|^p\right)\le C\Delta^{\frac{p}{2}}+C\Delta^{p}\le C\Delta^{\frac{p}{2}}.
\end{split}
\end{equation*}
Finally,
\begin{equation*}
\begin{split}
&\mathbb{E}\left(\sup\limits_{t_k\le t<t_{k+1}}|Y_\Delta(t)-\bar{Y}_\Delta(t)|^p\right)
\le C\Delta^{p}+\mathbb{E}\left(\sup\limits_{t_k\le t<t_{k+1}}|Z_\Delta(t)-\bar{Y}_\Delta(t)|^p\right)
\le C\Delta^{\frac{p}{2}}.
\end{split}
\end{equation*}
This completes the proof.
\end{proof}

\section{Strong convergence of modified $\theta$-EM Scheme}

Define the following stopping time
\begin{equation*}
\tau_\Delta=\inf\{t\in[0,T]: |X(t)|\ge g(\Delta)\}, \rho_\Delta=\inf\{t\in[0,T]: |Y_\Delta(t)|\ge g(\Delta)\}, \nu_\Delta=\tau_\Delta\wedge\rho_\Delta.
\end{equation*}
\begin{lem}\label{stop}
{\rm Under assumptions (A1)-(A2),  for any $T>0$, there exists a positive constant $C$ independent of $\Delta$ such that
\begin{equation*}
\mathbb{P}(\tau_\Delta\le T)\le \frac{C}{(g(\Delta))^p} \,
\mbox{ and
} \,
\mathbb{P}(\rho_\Delta\le T)\le \frac{C}{(g(\Delta))^p}.
\end{equation*}

}
\end{lem}
This Lemma can be shown by the Chebyshev inequality. Moreover we have
\begin{lem}\label{exact-num2}
{\rm
Let (A1)-(A2) hold. Then we have
\begin{equation*}
\begin{split}
&\mathbb{E}\left(\sup\limits_{0\le t \le T}|X(t\wedge\nu_\Delta)-Y_\Delta(t\wedge\nu_\Delta)|^2\right)
\le CM_{g(\Delta)}e^{M_{g(\Delta)}}\Delta^{\frac{1}{2}},
\end{split}
\end{equation*}
where $C$ is a positive constant independent of $\Delta$.
}
\end{lem}

\begin{proof}
Denote by $e(t)=X(t)-Z_\Delta(t)$. Applying the It\^{o} formula,
\begin{equation*}
\begin{split}
&|e(t\wedge\nu_\Delta)|^2\le|e(0)|^2+2\int_0^{t\wedge\nu_\Delta}\langle e(s),b(X(s))-b_\Delta(\bar{Y}_\Delta(s))\rangle\mbox{d}s\\
&+\int_0^{t\wedge\nu_\Delta}\|\sigma(X(s))-\sigma(\bar{Y}_\Delta(s))\|^2\mbox{d}s
+2\int_0^{t\wedge\nu_\Delta}\langle e(s),\sigma(X(s))-\sigma(\bar{Y}_\Delta(s))\mbox{d}W(s)\rangle\\
\le&|e(0)|^2+2\int_0^{t\wedge\nu_\Delta}\langle X(s)-\bar{Y}_\Delta(s),b(X(s))-b_\Delta(\bar{Y}_\Delta(s))\rangle\mbox{d}s+\int_0^{t\wedge\nu_\Delta}\|\sigma(X(s))-\sigma(\bar{Y}_\Delta(s))\|^2\mbox{d}s\\
&+2\int_0^{t\wedge\nu_\Delta}\langle e(s),\sigma(X(s))-\sigma(\bar{Y}_\Delta(s))\mbox{d}W(s)\rangle
+2\int_0^{t\wedge\nu_\Delta}\langle \theta b_\Delta(Y_\Delta(s))\Delta,b(X(s))-b_\Delta(\bar{Y}_\Delta(s))\rangle\mbox{d}s\\
&+2\int_0^{t\wedge\nu_\Delta}\langle \bar{Y}_\Delta(s)-Y_\Delta(s),b(X(s))-b_\Delta(\bar{Y}_\Delta(s))\rangle\mbox{d}s\\
=:&|e(0)|^2+\sum\limits_{i=1}^5e_i(t),
\end{split}
\end{equation*}
where $e(0)=\theta b_\Delta(x_0)\Delta$. Since for $0\le s\le t\wedge\nu_\Delta$, one has $|X(s)|\vee|\bar{Y}_\Delta(s)|\le g(\Delta)$. With the definition of $b_\Delta$, we then have $b_\Delta(\bar{Y}_\Delta(s))=b(\bar{Y}_\Delta(s))$ for $0\le s\le t\wedge\nu_\Delta$. It follows from (A2) and  Lemma \ref{ytytk} that
\begin{equation}\label{asedelta12}
\begin{split}
&\mathbb{E}\left(\sup\limits_{0\le u\le t}|e_1(u)+e_2(u)|\right)
\le CM_{g(\Delta)}\mathbb{E}\int_0^{t\wedge\nu_\Delta}|X(s)-\bar{Y}_\Delta(s)|^2\mbox{d}s\\
\le&CM_{g(\Delta)}\int_0^{t}\mathbb{E}\left(\sup\limits_{0\le u\le s}|X(u\wedge\nu_\Delta)-Y_\Delta(u\wedge\nu_\Delta)|^2\right)\mbox{d}s+CM_{g(\Delta)}\Delta.
\end{split}
\end{equation}
By the BDG inequality, we derive from Lemma \ref{ytytk} that
\begin{equation}\label{asedelta3}
\begin{split}
&\mathbb{E}\left(\sup\limits_{0\le u\le t}|e_3(u)|\right)
\le 12\mathbb{E}\left[\int_0^{t\wedge\nu_\Delta}|e(s)|^{2}\|\sigma(X(s))-\sigma(\bar{Y}_\Delta(s))\|^2\mbox{d}s\right]^{1/2}\mbox{d}s\\
\le&12\mathbb{E}\left[\sup\limits_{0\le u\le t}|e(s\wedge\nu_\Delta)|^{2}\int_0^{t\wedge\nu_\Delta}\|\sigma(X(s))-\sigma(\bar{Y}_\Delta(s))\|^2\mbox{d}s\right]^{1/2}\\
\le&\frac{1}{2}\mathbb{E}\left(\sup\limits_{0\le u\le t}|e(s\wedge\nu_\Delta)|^{2}\right)
+CM_{g(\Delta)}\int_0^{t}\mathbb{E}\left(\sup\limits_{0\le u\le s}|X(u\wedge\nu_\Delta)-Y_\Delta(u\wedge\nu_\Delta)|^2\right)\mbox{d}s+CM_{g(\Delta)}\Delta.
\end{split}
\end{equation}
With (A1) and the Young inequality, it follows from Lemma \ref{ytytk} that
\begin{equation}\label{asedelta4}
\begin{split}
\mathbb{E}\left(\sup\limits_{0\le u\le t}|e_4(u)|\right)\le & 2\theta\Delta\int_0^{t}\left[\mathbb{E}|b(Y_\Delta(s\wedge\nu_\Delta))|^2\right]^{\frac{1}{2}}\left[\mathbb{E}|b(X(s\wedge\nu_\Delta))-b(\bar{Y}_\Delta(s\wedge\nu_\Delta))|^{2}\right]^{\frac{1}{2}}\mbox{d}s\le C\Delta.
\end{split}
\end{equation}
Now, using (A1),  Lemmas \ref{exactpm} and  \ref{ytytk} again, we get
\begin{equation}\label{asedelta5}
\begin{split}
\mathbb{E}\left(\sup\limits_{0\le u\le t}|e_5(u)|\right)\le&2\int_0^{t}\left[\mathbb{E}|\bar{Y}_\Delta(s\wedge\nu_\Delta)-Y_\Delta(s\wedge\nu_\Delta)|^{2}\right]^{\frac{1}{2}}
\left[\mathbb{E}|b(X(s\wedge\nu_\Delta))-b(\bar{Y}_\Delta(s\wedge\nu_\Delta))|^{2}\right]^{\frac{1}{2}}\\
\le& C\Delta^{\frac{1}{2}}.
\end{split}
\end{equation}
Combining \eqref{asedelta12}-\eqref{asedelta5}, we obtain
\begin{equation*}
\begin{split}
&\mathbb{E}\left(\sup\limits_{0\le u\le t}|e(u\wedge\nu_\Delta)|^2\right) \le CM_{g(\Delta)}\Delta+C\Delta^{\frac{1}{2}}+CM_{g(\Delta)}\int_0^{t}\mathbb{E}\left(\sup\limits_{0\le u\le s}|X(u\wedge\nu_\Delta)-Y_\Delta(u\wedge\nu_\Delta)|^2\right)\mbox{d}s.
\end{split}
\end{equation*}
Hence, Lemma \ref{lemmono} and Lemma \ref{ytytk} lead to
\begin{equation*}
\begin{split}
&\mathbb{E}\left(\sup\limits_{0\le u \le t}|X(u\wedge\nu_\Delta)-Y_\Delta(u\wedge\nu_\Delta)|^2\right)\\
\le&\mathbb{E}\left(\sup\limits_{0\le u \le t}|e(u\wedge\nu_\Delta)|^2\right)+\theta^2\Delta^2\mathbb{E}\left(\sup\limits_{0\le u \le t}|b_\Delta(Y_\Delta(u\wedge\nu_\Delta))|^2\right)\\
\le&CM_{g(\Delta)}\int_0^{t}\mathbb{E}\left(\sup\limits_{0\le u\le s}|X(u\wedge\nu_\Delta)-Y_\Delta(u\wedge\nu_\Delta)|^2\right)\mbox{d}s
+CM_{g(\Delta)}\Delta^{\frac{1}{2}}.
\end{split}
\end{equation*}
Finally, the Gronwall inequality leads to the desired result.

\end{proof}

\begin{thm}\label{theorem1}
{\rm Let (A1)-(A2) hold. Then for any $\epsilon>0$, there exists a $\Delta^*$ such that
\begin{equation*}
\mathbb{E}\left(\sup\limits_{0\le t\le T}|X(t)-Y_\Delta(t)|^2\right)<\epsilon, \mbox{ whenever } \Delta\in (0,  \Delta^*).
\end{equation*}
}
\end{thm}
\begin{proof}
By the Young inequality, we derive that for any $\eta>0$,
\begin{equation}\label{eTq22}
\begin{split}
&\mathbb{E}\left(\sup\limits_{0\le t\le T}|X(t)-Y_\Delta(t)|^2\right)=\mathbb{E}\left({\bf I}_{\{\tau_\Delta>T,\rho_\Delta>T\}}\sup\limits_{0\le t\le T}|X(t)-Y_\Delta(t)|^2\right)\\
&+\frac{2\eta}{p}\mathbb{E}\left(\sup\limits_{0\le t\le T}|X(t)-Y_\Delta(t)|^{p}\right)+\frac{p-2}{p\eta^{\frac{2}{p-2}}}\mathbb{P}(\tau_\Delta\le T~{\rm or}~\rho_\Delta\le T).
\end{split}
\end{equation}
By Lemmas \ref{exactpm}, \ref{ytytk} and \ref{stop}, we derive
\begin{equation*}
\begin{split}
\mathbb{E}\left(\sup\limits_{0\le t\le T}|X(t)-Y_\Delta(t)|^2\right)\le&\mathbb{E}\left(\sup\limits_{0\le t\le T}|X(t\wedge\nu_\Delta)-Y_\Delta(t\wedge\nu_\Delta)|^2\right)+\frac{2\eta C}{p}+\frac{2(p-2)C}{p\eta^{\frac{2}{p-2}}(g(\Delta))^{p}}.
\end{split}
\end{equation*}
For any given $\epsilon>0$, we choose $\eta$  small enough such that
 \begin{equation*}
\begin{split}
\frac{2\eta C}{p}<\frac{\epsilon}{3},
\end{split}
\end{equation*}
and $\Delta$ small enough such that
 \begin{equation*}
\frac{2(p-2)C}{p\eta^{\frac{2}{p-2}}(g(\Delta))^{p}}<\frac{\epsilon}{3} \mbox{ and }
M_{g(\Delta)}e^{M_{g(\Delta)}}\Delta^{\frac{1}{2}}<\frac{\epsilon}{3}.
\end{equation*}
Thus, we get
\begin{equation*}
\begin{split}
\mathbb{E}\left(\sup\limits_{0\le t\le T}|X(t)-Y_\Delta(t)|^2\right)<\epsilon~~\mbox{  whenever } \Delta \mbox{ is sufficiently small}.
\end{split}
\end{equation*}
\end{proof}

\begin{rem}
{\rm Let's illustrate the result of Theorem \ref{theorem1} through Example \ref{example1}. For $\varepsilon\in(0,1)$, let $g(\Delta)=\frac{1}{2}(\ln\Delta^{-\frac{\varepsilon}{8}}-1)$, it is obvious that $g(\Delta)$ is strictly decreasing and $g(\Delta)\rightarrow\infty$ as $\Delta\rightarrow 0$. Moreover, we have $M_{g(\Delta)}e^{M_{g(\Delta)}}\Delta^{\frac{1}{4}}\le1$ and $(g(\Delta))^l\Delta\le1$ for some $\Delta$. Thus, by Theorem \ref{theorem1}, the modified $\theta$-EM scheme converges strongly to exact solution of the local one-sided Lipschitz equation in Example \ref{example1}, that is,
\begin{equation*}
\lim\limits_{\Delta\rightarrow0}\mathbb{E}\left(\sup\limits_{0\le t\le T}|X(t)-Y_\Delta(t)|^2\right)=0.
\end{equation*}

}
\end{rem}

\begin{rem}
Hutzenthaler et al. \cite{hjk11} pointed out that the absolute moments of EM scheme at a finite time could diverge to infinity for SDEs with one-sided Lipschitiz and superlinear growing coefficients, which implied that the EM scheme would not converge in the strong sense to the exact solution in this case. In our paper, the superlinear property is avoided in the estimation of $|y_{t_k}|^p$ (see Lemma \ref{pmoment} for more details) for $\theta\in[1/2,1]$, and it is shown that for SDEs with one-sided Lipschitiz and superlinear growing coefficients, our modified $\theta$-EM scheme with $\theta\in[1/2,1]$ converges to the exact solution.
\end{rem}

\begin{rem}
Under the local one-sided Lipschitz condition for the drift coefficients, we have shown that the modified $\theta$-EM converges to the true solution, however, we have not obtained the rate of the convergence, which could be an open problem.
\end{rem}

\end{document}